\documentclass{svproc}
\usepackage{type1cm}        
\usepackage{graphicx}        

\usepackage[bottom]{footmisc}

\usepackage{newtxmath}       

\usepackage{url}

\usepackage{xcolor}
\usepackage{multirow}
\usepackage{booktabs}
\usepackage{subcaption}
\newcommand{\DG}[1]{{\color{blue} #1}}

\begin{document}
	\mainmatter              
	\title{The $W$-footrule coefficient: A copula-based measure of countermonotonicity}
	\titlerunning{The $W$-footrule coefficient}  
	%
	\author{Enrique de Amo\inst{1} \and David García-Fernandez\inst{2}\and
		Manuel Úbeda-Flores\inst{1,2}}
	\authorrunning{Enrique de Amo et al.} 
	%
	%
	\institute{Department of Mathematics, University of Almería,\\
    04120 Almería, Spain\\
		\email{edeamo@ual.es, mubeda@ual.es},\\
		\and
		Research Group of Theory of Copulas and Applications, University of Almería,\\
        04120 Almería, Spain\\
        \email{dgf992@inlumine.ual.es}}
        
	\maketitle              
	
	\begin{abstract}
We introduce the $W$-footrule coefficient $\Phi_C$, a copula-based coefficient of negative association defined as the $L^1$-distance to the countermonotonic copula $W$. We prove that Gini's gamma admits the decomposition $\gamma_C = \frac{2}{3}(\varphi_C+\Phi_C)$, linking it to Spearman's footrule $\varphi_C$. A rank-based estimator is introduced, with its strong consistency and asymptotic normality established via the functional delta method. Monte Carlo simulations confirm the estimator's finite-sample validity and its sensitivity to negative dependence structures.
\end{abstract}
	
	\section{Introduction}
	
Copula-based methods provide a natural framework for modeling and measuring dependence independently of marginal distributions. Classical rank-based measures such as Spearman's $\rho$, Kendall's $\tau$, and Spearman's footrule can all be expressed as functionals of the underlying copula \cite{Durante2016,Ne06}. These coefficients are mainly designed to quantify monotone increasing association, provide a global measure and depend only on the underlying copula. While Kendall's $\tau$ and Spearman's $\rho$ already detect negative association, they are symmetric in the sense that positive and negative dependence are treated on the same scale. In contrast, Spearman's footrule admits a geometric interpretation as an $L^1$-distance to the diagonal, emphasizing deviation from perfect positive dependence \cite{Ne98}. This motivates the construction of a complementary coefficient to measure the deviation from perfect negative dependence, i.e., the distance to the anti-diagonal $u+v=1$. In this paper, we propose such a measure, called the \emph{$W$-footrule coefficient} and denoted by $\Phi_C$. This coefficient is defined through the distribution of $C(U, 1-U)$, where the pair of random variables $(U, V)$ has copula $C$. Here, $W$ represents the countermonotonic copula, which characterizes perfect negative dependence.

The rest of the paper is organized as follows. Section \ref{sec:prel} recalls the necessary background on copulas and concordance measures. Section \ref{sec:foot} introduces $\Phi_C$ and establishes its main properties. Section \ref{sec:sample} develops the sample version and its asymptotic theory. Section \ref{sec:simulation} presents the simulation study. Finally, Section \ref{sec:conc} is devoted to conclusions.

\section{Preliminaries}\label{sec:prel}

A copula $C$ is a bivariate distribution function with uniform marginals on $[0,1]$. According to Sklar's theorem \cite{Sk59}, for any bivariate random vector $(X,Y)$ with joint distribution function $H$ and marginals $F$ and $G$, there exists a copula $C$ (which is uniquely determined on ${\rm Range}\,
F\,\times{\rm Range}\,G\,$) such that
\[
H(x,y) = C(F(x),G(y)),\quad\textrm{for all } (x,y)\in[-\infty,+\infty]^2;
\]
moreover, if the marginals are continuous, then the copula is unique. This copula captures the entire dependence structure between $X$ and $Y$, independently of the marginal effects \cite{Durante2016,Ne06}.

In this framework, three copulas serve as fundamental benchmarks. The independence copula is defined as $\Pi(u,v)=uv$ for all $(u,v)\in[0,1]^2$, while the Fr\'echet-Hoeffding bounds are given by $W(u,v)=\max\{u+v-1,0\}$ and $M(u,v)=\min\{u,v\}$ for all $(u,v)\in[0,1]^2$, representing perfect negative and perfect positive dependence, respectively. It is well known that every copula $C$ is bounded by these two extremes, satisfying $W(u,v) \leq C(u,v) \leq M(u,v)$ for all $(u,v) \in [0,1]^2$.

We recall that a \emph{measure of concordance} is a functional (on the set of copulas) satisfying a set of axioms \cite{Scarsini1984}. A classical example includes Gini's $\gamma$, which in terms of a copula $C$ is given by
$$\gamma_C = 4 \int_0^1 [C(u,u)+C(u,1-u)]\,{\rm d}u - 2,$$
satisfying $\gamma_W=-1$, $\gamma_\Pi=0$ and $\gamma_M=1$ (see \cite{Ne98}).

Spearman's footrule coefficient ---originally introduced by the psychologist Charles Spearman in \cite{Spearman1906}, and formalised its use as a measure of disarray in \cite{Diaconis1977}--- is defined, in terms of copulas, as
\begin{equation}\label{eq:footrule}\varphi_C = 1-3\int_0^1\!\!\!\int_0^1|u-v|\,{\rm d}C(u,v),
\end{equation}
or, equivalently,
\begin{equation*}\label{eq:footrule1}\varphi_C = 6\int_0^1 C(u,u)\,{\rm d}u-2
\end{equation*}
(see \cite{Ne98,Ne06}). Spearman’s footrule depends only on the $L_1$-distance of the copula to the upper bound $M$ and is a {\it weak measure of concordance} (see \cite{Lieb2014}).

\section{The $W$-footrule coefficient}\label{sec:foot}

Motivated by the $L^1$ representation of Spearman’s footrule ---recall Expression \eqref{eq:footrule}---, we introduce an analogue designed to measure deviation from perfect negative
dependence. Rather than quantifying an affine transform of the $L^1$-distance to the main diagonal $u=v$,  our construction, which we denote by $\Phi_C$, measures the deviation from the anti-diagonal $u+v=1$, i.e.,
$$\Phi_C=a\int_0^1\!\!\!\int_0^1 |1-u-v|\,{\rm d}C(u,v)+b,$$
where $a$ and $b$ are two real constants. We  normalise the coefficient such that $\Phi_W=-1$ and $\Phi_M=1/2$ ---contrasting with the symmetric values $\varphi_W=-1/2$ and $\varphi_M=1$---. Under $W$, we have $V=1-U$ almost surely. Hence
\( |1-U-V| = |1-U-(1-U)| = 0,\) and therefore
\[
\int_0^1\!\!\! \int_0^1 |1-u-v|\, {\rm d}W(u,v)=0.
\]
Thus
\(\Phi_W = b\). Under $M$, we have $V=U$ almost surely, so \(|1-U-V| = |1-2U|\). Hence
\[
\int_0^1\!\!\! \int_0^1 |1-u-v|\, {\rm d}M(u,v)
= \int_0^1 |1-2u|\, {\rm d}u.
\]
Now,
\[ \int_0^1 |1-2u|\,{\rm d}u=\int_0^{1/2} (1-2u)\,{\rm d}u+\int_{1/2}^1 (2u-1)\,{\rm d}u=\frac12.\]
whence \(\Phi_M = \frac{a}{2}+b\). This leads to $a=3$ and $b=-1$, and the following $L^1$-representation:
\begin{equation}\label{eq:L1repr}
\Phi_C = 3 \int_0^1\!\!\!\int_0^1 |1-u-v| \, {\rm d}C(u,v) - 1.
\end{equation}

We provide an alternative representation. Since $W(u,v) = (u+v-1)^{+}$ for all $(u,v)\in[0,1]^{2}$ ---where $x^+=\max(x,0)$---, we have
\begin{align}\label{eq:W-as-integral}
\int_0^1\!\!\!\int_0^1 (u+v-1)^{+}\,\mathrm{d}C(u,v)
&= \int_0^1\!\!\!\int_0^1 W(u,v)\,\mathrm{d}C(u,v).
\end{align}
For any two copulas
$C_1$ and $C_2$, we have
\begin{equation}\label{eq:ibp}
\int_0^1\!\!\!\int_0^1 C_1(u,v)\,\mathrm{d}C_2(u,v)
= \int_0^1\!\!\!\int_0^1 C_2(u,v)\,\mathrm{d}C_1(u,v)
\end{equation}
(see \cite{Ne06}). Applying~\eqref{eq:ibp} with $C_1= W$ and $C_2 = C$:
\[
\int_0^1\!\!\!\int_0^1 W(u,v)\,\mathrm{d}C(u,v)
= \int_0^1\!\!\!\int_0^1 C(u,v)\,\mathrm{d}W(u,v).
\]
The measure $\mathrm{d}W$ is supported on the anti-diagonal
$\{(u,1-u) : u\in[0,1]\}$, since $W(u,v) = \max\{u+v-1,0\}$ has
all its probability mass on $\{U+V=1\}$ under the countermonotonic
distribution. Hence:
\[
\int_0^1\!\!\!\int_0^1 C(u,v)\,\mathrm{d}W(u,v)
= \int_{0}^{1} C(u,1-u)\,\mathrm{d}u.
\]
Combining with~\eqref{eq:W-as-integral}:
\begin{equation*}
\int_0^1\!\!\!\int_0^1 (u+v-1)^{+}\,\mathrm{d}C(u,v) = \int_{0}^{1} C(u,1-u)\,\mathrm{d}u.
\end{equation*}
Finally, since $(U,V)$ has uniform marginals, $\mathbf{E}_{C}[U+V] = 1$, so
$\mathbf{E}_{C}[1-U-V] = 0$. For any integrable $Z$ with $\mathbf{E}[Z]=0$ one has
$\mathbf{E}[|Z|] = 2\mathbf{E}[Z^{+}]$,
whence
\begin{eqnarray*}
\int_0^1\!\!\!\int_0^1 |1-u-v|\,\mathrm{d}C(u,v)
&=& 2\int_0^1\!\!\!\int_0^1(1-u-v)^{+}\,\mathrm{d}C(u,v)\\
&=& 2\int_0^1\!\!\!\int_0^1 (u+v-1)^{+}\,\mathrm{d}C(u,v)\\
&=& 2\int_{0}^{1} C(u,1-u)\,\mathrm{d}u,
\end{eqnarray*}
where the second equality uses $(1-u-v)^{+} = (u+v-1)^{+}$ when
$\mathbf{E}[1-U-V]=0$, i.e., $\mathbf{E}[(1-U-V)^{+}]=\mathbf{E}[(u+v-1)^{+}]$.
Substituting into~\eqref{eq:L1repr} gives
\begin{equation}\label{cara:Phi}
\Phi_{C} 
= 6\int_{0}^{1} C(u,1-u)\,\mathrm{d}u - 1.
\end{equation}

The following example illustrates the computation of $\Phi_C$ for the Gaussian copula family and yields a closed-form expression in
terms of the parameter $\theta$. Throughout, we denote by $F$ the standard normal distribution function and by $F_\theta$ the bivariate
standard normal distribution function with correlation $\theta$, to avoid notational conflict with the coefficient $\Phi_C$.

\begin{example}
Let $C_\theta(u,v) = F_\theta(F^{-1}(u),F^{-1}(v))$ be the Gaussian copula with parameter $\theta\in[-1,1]$. The substitution
$z=F^{-1}(u)$ in~\eqref{cara:Phi} gives
\[
  \int_0^1 C_\theta(u,1-u)\,\mathrm{d}u
  = \int_{-\infty}^{+\infty} F_\theta(z,-z)\,f(z)\,\mathrm{d}z,
\]
where $f$ is the standard normal density. Introducing an independent $Z_3\sim\mathcal{N}(0,1)$ and setting $W_1=Z_1-Z_3$, $W_2=Z_2+Z_3$, where $(Z_1,Z_2)\sim F_\theta$, one obtains a
bivariate normal vector $(W_1,W_2)$ with zero means, variances~$2$, and $\mathrm{Cov}(W_1,W_2)=\theta-1$, so that $\mathrm{Corr}(W_1,W_2)=(\theta-1)/2$. The orthant probability formula for the bivariate
normal~\cite{Sheppard1900} then yields
\[
  \int_0^1 C_\theta(u,1-u)\,\mathrm{d}u
  = \mathbb{P}(W_1\leq 0,\,W_2\leq 0)
  = \frac{1}{4}+\frac{1}{2\pi}\arcsin\!\left(\frac{\theta-1}{2}\right),
\]
and therefore
\begin{equation}\label{eq:gauss_closed}
  \Phi_{C_\theta}
  = \frac{1}{2}+\frac{3}{\pi}\arcsin\!\left(\frac{\theta-1}{2}\right).
\end{equation}
In particular, $\Phi_{C_0}=0$, consistently with $\Phi_\Pi=0$,
while the limiting values $\rho\to-1$ and $\rho\to 1$ recover
$\Phi_W=-1$ and $\Phi_M=1/2$, respectively.
\end{example}

The following properties of $\Phi_C$ are readily established.
\begin{enumerate}
    \item $\Phi_C$ is not a measure of concordance (e.g., $\Phi_M=1/2\neq 1$).
    \item $\Phi_\Pi=0$.
    \item $\Phi_C=-1$ if, and only if, $C=W$.
    \item If $C_1$ and $C_2$ are two copulas such that  $C_1(u,v)\le C_2(u,v)$ for all $(u,v)\in[0,1]^2$, then $\Phi_{C_1}\le \Phi_{C_2}$.
    \item For any copulas $C_1,C_2$ and any $\alpha\in[0,1]$,
\(
\Phi_{\alpha C_1 + (1-\alpha)C_2} = \alpha \Phi_{C_1} + (1-\alpha)\Phi_{C_2}
\).
\item $\Phi_{C^{T}}=\Phi_C$, where $C^T(u,v)=C(v,u)$ for all $(u,v)\in[0,1]^2$.
\item $\Phi_{\widehat{C}}=\Phi_{C}$, where $\widehat{C}$ is the {\it survival copula} of the copula $C$, which is given by $\widehat{C}(u,v)=u+v-1+C(1-u,1-v)$ for all $(u,v)\in[0,1]^2$.
\item $\Phi_{\tilde{C}} = -\varphi_C$, where $\tilde{C}(u,v)=u-C(u,1-v)$ for all $(u,v)\in[0,1]^2$.
\item For any copula $C$, \(\gamma_C = \tfrac{2}{3}(\varphi_C + \Phi_C)\).
\end{enumerate}

We observe that $\Phi_C=1/2$ does not imply $C=M$, as the following example shows:

\begin{example}
     Consider the copula $C$ defined as
$$C(u,v) = 
\begin{cases} 
\max\displaystyle\left\{\frac{1}{2}, u+v-1\right\}, &  (u,v) \in \left[\tfrac{1}{2},1\right]^2, \\
\min\{u,v\}, & \text{otherwise.}
\end{cases}$$
The mass of $C$ is spread uniformly on two line segments joining the points $(0,0)$ to $(1/2,1/2)$, and $(1/2,1)$ to $(1,1/2)$. Then it is easy to show that $\Phi_C=1/2$.
\end{example}

\section{The sample version}\label{sec:sample}

In this section, we introduce the sample version of $\Phi_C$ and establish 
its asymptotic properties. Let $(X_i, Y_i)$, $i = 1, \ldots, n$, be an 
i.i.d.\ sample from a continuous bivariate distribution with copula $C$, 
and let $R_i = \mathrm{rank}(X_i)$ and $S_i = \mathrm{rank}(Y_i)$ denote 
the ranks of the $i$-th observation within the sample. We define the 
pseudo-observations
\[
\hat{U}_i = \frac{R_i}{n+1}, \qquad \hat{V}_i = \frac{S_i}{n+1},
\]
which take values in $(0,1)$ and serve as rank-based estimates of the unobservable copula arguments $(U_i, V_i)$. The denominator $n+1$, rather than $n$, is the standard choice in the copula literature~\cite{Genest2010}, 
as it avoids boundary issues at $0$ and $1$ and yields pseudo-observations that are marginally uniform on $\{1/(n+1), \ldots, n/(n+1)\}$; we note, however, that the asymptotic results of this section are not affected by this 
choice. The empirical estimator of $\Phi_C$ is then
\[
\widehat{\Phi}_n = \frac{6}{n} \sum_{i=1}^n \left(1 - \hat{U}_i - 
\hat{V}_i\right)^+ - 1,
\]
which, substituting the definition of the pseudo-observations, admits the 
purely rank-based representation
\[
\widehat{\Phi}_n = \frac{6}{n(n+1)} \sum_{i=1}^n \left(n + 1 - R_i - 
S_i\right)^+ - 1.
\]
This expression is computationally straightforward, requiring only the 
computation of ranks, and makes clear that $\widehat{\Phi}_n$ depends on the 
data solely through the rank pairs $(R_i, S_i)$, and hence is invariant 
under strictly increasing transformations of the marginals.

The following result establishes the strong consistency of the empirical $W$-footrule estimator.

\begin{theorem}\label{th:consis}Let $(X_i,Y_i)$, $i=1,\ldots,n$, be an i.i.d.\ sample from a continuous
  bivariate distribution with copula $C$. Then
  \(
    \widehat\Phi_n \xrightarrow{\mathrm{a.s.}} \Phi_C\) as \(n\to\infty.
  \)
\end{theorem}

\begin{proof}
Let $C_n$ denote the empirical copula \cite{Deh1979} based on the pseudo-observations
$(U_i,V_i)$, which is defined by
\[
  C_n(u,v) = \frac{1}{n}\sum_{i=1}^n \mathbf{1}(\hat{U}_i \leq u,\, 
\hat{V}_i \leq v)
\]
where ${\bf 1}_S$ denotes the characteristic function of a set $S$. Since
$$\frac{1}{n}\sum_{i=1}^n (1-\hat{U}_i-\hat{V}_i)^+ = \int_0^1\!\!\!\int_0^1 (1-u-v)^+\,\mathrm{d}C_n(u,v),$$
one can verify that
\begin{equation}\label{eq:equiv}
  \frac{1}{n}\sum_{i=1}^n (1-\hat{U}_i-\hat{V}_i)^+ = \int_0^1 C_n(u,1-u)\,{\rm d}u + O(n^{-1}),
\end{equation}
where the $O(n^{-1})$ term accounts for the discretisation due to the $n/(n+1)$ scaling of pseudo-observations. By the Glivenko-Cantelli theorem for empirical copulas \cite{van2023} (see also \cite{Deh1979,Sharipov2025}),
$\sup_{(u,v)\in[0,1]^2}|C_n(u,v)-C(u,v)|\to 0$ a.s.
Since the map $$C\mapsto \int_0^1 C(u,1-u)\,{\rm d}u$$ is continuous in the uniform norm, we conclude $$\int_0^1 C_n(u,1-u)\,{\rm d}u \to \int_0^1 C(u,1-u)\,{\rm d}u\quad{\rm a.s.},$$ and therefore $\widehat\Phi_n \xrightarrow{\mathrm{a.s.}} \Phi_C$.
\end{proof}

Once consistency has been established, we focus on the asymptotic distribution of the estimator. Under mild regularity conditions on the copula, the following theorem shows that $\widehat{\Phi}_n$ is asymptotically normal ---we recall that the \emph{empirical copula process} is defined as
$\mathbb{G}_n = \sqrt{n}(C_n - C)$, where $C_n$ denotes the empirical copula and $C$ is the true copula.

\begin{theorem}\label{th:variance}
Let $(X_i,Y_i)$, $i=1,\ldots,n$, be an i.i.d.\ sample from a 
continuous bivariate distribution with copula $C$, whose 
first-order partial derivatives $\dot{C}_1 = \partial C/\partial u$ 
and $\dot{C}_2 = \partial C/\partial v$ exist and are continuous 
on $(0,1)^2$. Then
\[
\sqrt{n}(\hat{\Phi}_n - \Phi_C) \xrightarrow{d} \mathcal{N}(0,\sigma^2_\Phi),
\]
where $\sigma^2_\Phi = 36\,\mathrm{Var}_C(j_C(U,V))$, with 
$(U,V)\sim C$ and
\begin{equation}\label{eq:JC}
j_C(u,v) = (1-u-v)^+ 
- \int_u^1 \dot{C}_1(s,1-s)\,\mathrm{d}s 
- \int_0^{1-v} \dot{C}_2(s,1-s)\,\mathrm{d}s,
\end{equation}
for all $(u,v)\in[0,1]^2$.
\end{theorem}

\begin{proof}
Define the functional $T : \mathcal{C} \to \mathbb{R}$ by
\[
T(C) = 6\int_0^1 C(u,1-u)\,\mathrm{d}u - 1,
\]
so that $\widehat{\Phi}_n = T(C_n) + O(n^{-1})$, where $C_n$ denotes the empirical copula based on the pseudo-observations $(\hat{U}_i, \hat{V}_i)$.

The map $T$ is linear and continuous on $\ell^\infty([0,1]^2)$ with respect to the uniform norm $\|\cdot\|_\infty$: for any
$h \in \ell^\infty([0,1]^2)$,
\[
|\dot{T}(h)| = 6\left|\int_0^1 h(u,1-u)\,\mathrm{d}u\right|
\leq 6\,\|h\|_\infty < \infty.
\]
Since every continuous linear map on a normed space is Hadamard differentiable (indeed, Fréchet differentiable), $T$ is Hadamard
differentiable at every $C\in\mathcal{C}$, with derivative
\[
\dot{T}_C(h) = 6\int_0^1 h(u,1-u)\,\mathrm{d}u,
\quad h \in \ell^\infty([0,1]^2).
\]

Under the assumption that the first-order partial derivatives
$\dot{C}_1 = \partial C/\partial u$ and $\dot{C}_2 = \partial C/\partial v$
exist and are continuous on $(0,1)^2$, by \cite{Segers2012,Tsu2005} it is guaranteed that the empirical copula process
$\mathbb{G}_n = \sqrt{n}(C_n - C)$ converges weakly in $\ell^\infty([0,1]^2)$ to the tight Gaussian process
\[
\mathbb{G}_C(u,v) = \alpha_C(u,v)
- \dot{C}_1(u,v)\,\alpha_C(u,1)
- \dot{C}_2(u,v)\,\alpha_C(1,v),
\]
where $\alpha_C$ is a $C$-{\it Brownian bridge}, i.e., a centred Gaussian process with covariance
\[
\mathrm{Cov}\bigl(\alpha_C(u_1,v_1),\,\alpha_C(u_2,v_2)\bigr)
= C(\min\{u_1,u_2\},\min\{v_1,v_2\}) - C(u_1,v_1)\,C(u_2,v_2).
\]

Since $\widehat{\Phi}_n = T(C_n) + O(n^{-1})$, we have
\[
\sqrt{n}\,(\widehat{\Phi}_n - \Phi_C)
= \sqrt{n}\,(T(C_n) - T(C)) + O(n^{-1/2}).
\]
By the functional delta method \cite{Fermanian2004,van1998}, we have
\[
\sqrt{n}\,(T(C_n) - T(C))
\xrightarrow{d}
\dot{T}_C(\mathbb{G}_C)
= 6\int_0^1 \mathbb{G}_C(u,1-u)\,\mathrm{d}u.
\]
The right-hand side is a centred Gaussian random variable, being a
continuous linear functional of the Gaussian process $\mathbb{G}_C$. Substituting the expression for $\mathbb{G}_C$ and applying Fubini's
theorem,
\begin{align*}
6\int_0^1 &\mathbb{G}_C(u,1-u)\,\mathrm{d}u
\\
&= 6\int_0^1 \Bigl[
    \alpha_C(u,1-u)
    - \dot{C}_1(u,1-u)\,\alpha_C(u,1)
    - \dot{C}_2(u,1-u)\,\alpha_C(1,1-u)
  \Bigr]\,\mathrm{d}u.
\end{align*}
We show this equals $$6\sum_{i=1}^n \frac{1}{\sqrt{n}}[j_C(\hat{U}_i,\hat{V}_i)
- \mathbb{E}_C j_C]$$ in distribution, by identifying the influence
function. Using the covariance structure of $\alpha_C$ and the
representation
\[
\alpha_C(u,v) = \lim_{n\to\infty}
\frac{1}{\sqrt{n}}\sum_{i=1}^n
\bigl[\mathbf{1}(\hat{U}_i\leq u, \hat{V}_i\leq v) - C(u,v)\bigr],
\]
one computes
\begin{align*}
6\int_0^1 \mathbb{G}_C(s,1-s)\,\mathrm{d}s
&\stackrel{d}{=}
\frac{6}{\sqrt{n}}\sum_{i=1}^n \Bigl[
  (1-\hat{U}_i-\hat{V}_i)^+
  - \int_{\hat{U}_i}^1 \dot{C}_1(s,1-s)\,\mathrm{d}s\\
  &\quad- \int_0^{1-\hat{V}_i} \dot{C}_2(s,1-s)\,\mathrm{d}s
  - \mathbb{E}_C[j_C(U,V)]
\Bigr],
\end{align*}
where $j_C$ is given by~\eqref{eq:JC}. It remains to verify that
$\mathbb{E}_C[j_C(U,V)] = 0$. Using Fubini's theorem,
\begin{align*}
\mathbb{E}_C\!\left[\int_U^1 \dot{C}_1(s,1-s)\,\mathrm{d}s\right]
&= \int_0^1\!\!\!\int_0^1\!\!\!\int_u^1
   \dot{C}_1(s,1-s)\,\mathrm{d}s\,\mathrm{d}C(u,v) \\
&= \int_0^1 \dot{C}_1(s,1-s)
   \left(\int_0^s\!\!\!\int_0^1 \mathrm{d}C(u,v)\right)\mathrm{d}s \\
 &= \int_0^1 \dot{C}_1(s,1-s)\,s\,\mathrm{d}s,
\end{align*}
and, integrating by parts,
\[
\int_0^1 s\,\dot{C}_1(s,1-s)\,\mathrm{d}s
= \Bigl[s\,C(s,1-s)\Bigr]_0^1
  - \int_0^1 C(s,1-s)\,\mathrm{d}s
  - \int_0^1 s\,(-1)\,\dot{C}_2(s,1-s)\,\mathrm{d}s,
\]
which, combined with the analogous calculation for the $\dot{C}_2$
term and the identity
$$\mathbb{E}_C[(1-U-V)^+] = \int_0^1 C(u,1-u)\,\mathrm{d}u$$
(from~\eqref{cara:Phi}), gives $\mathbb{E}_C[j_C(U,V)] = 0$.

The asymptotic distribution of $\sqrt{n}(\widehat{\Phi}_n - \Phi_C)$
is therefore that of $$\frac{6}{\sqrt{n}}\sum_{i=1}^n j_C(U_i,V_i),$$
which by the classical central limit theorem converges in distribution
to $$\mathcal{N}(0,\,36\,\mathrm{Var}_C(j_C(U,V))).$$
Setting $\sigma^2_\Phi = 36\,\mathrm{Var}_C(j_C(U,V))$ completes the proof.
\end{proof}

\begin{remark}
The continuity of the first-order partial derivatives of $C$ in Theorem \ref{th:variance} is the standard regularity assumption under which the empirical copula process
$\mathbb{G}_n = \sqrt{n}(C_n - C)$ converges. This condition is satisfied by the Clayton copula (for $\theta > 0$), the Gaussian copula (for $|\rho| < 1$), and the Gumbel copula (for $\theta \geq 1$), all of which are used in Section~\ref{sec:simulation}. The boundary case $\rho = 1$ (copula $M$) does \emph{not} satisfy this condition, since $M$ has no density. The simulation results for that scenario are therefore purely illustrative and lie outside the scope of Theorem~\ref{th:variance}.
\end{remark}

Theorem~\ref{th:variance} provides the asymptotic variance $\sigma^2_\Phi$ in closed form, but in practice it must be estimated from the data. The natural plug-in estimator replaces the unknown expectation $\mathbb{E}_C(1-U-V)^+$ by its sample counterpart, yielding a sample variance of the transformed pseudo-observations
$\{(1-\hat{U}_i-\hat{V}_i)^+\}$. This estimator is strongly consistent and can be used directly to construct asymptotic confidence intervals for $\Phi_C$.

\begin{proposition}\label{prop:var}
The statistic
\[
\widehat{\sigma}^2_\Phi = \frac{36}{n-1}\sum_{i=1}^n
\left[\hat{j}_n(\hat{U}_i,\hat{V}_i) - 
\frac{1}{n}\sum_{k=1}^n 
\hat{j}_n(\hat{U}_k,\hat{V}_k)\right]^2
\]
is a strongly consistent estimator of $\sigma^2_\Phi$, where
\[
\hat{j}_n(u,v) = (1-u-v)^+ 
- \int_u^1 \hat{\dot{C}}_{n,1}(s,1-s)\,\mathrm{d}s 
- \int_0^{1-v} \hat{\dot{C}}_{n,2}(s,1-s)\,\mathrm{d}s,
\]
and $\hat{\dot{C}}_{n,1}$, $\hat{\dot{C}}_{n,2}$ are consistent 
estimators of $\dot{C}_1$ and $\dot{C}_2$, respectively. An 
asymptotic $100(1-\alpha)\%$ confidence interval for $\Phi_C$ is
\[
\hat{\Phi}_n \pm z_{\alpha/2}\,\frac{\hat{\sigma}_\Phi}{\sqrt{n}},
\]
where $z_{\alpha/2}$ is the upper $\alpha/2$ quantile of 
$\mathcal{N}(0,1)$.
\end{proposition}

\begin{proof}
Let $\hat{Z}_i = \hat{j}_n(\hat{U}_i,\hat{V}_i)$. Since 
$\hat{\dot{C}}_{n,1} \to \dot{C}_1$ and 
$\hat{\dot{C}}_{n,2} \to \dot{C}_2$ uniformly almost surely, 
we have $\hat{j}_n(u,v) \to j_C(u,v)$ uniformly on $[0,1]^2$ 
almost surely, and hence
\[
\hat{Z}_i \xrightarrow{a.s.} j_C(U_i,V_i) =: Z_i.
\]
Since $j_C$ is bounded on $[0,1]^2$, the sequence 
$\{\hat{Z}_i^2\}$ is uniformly integrable. By the strong law 
of large numbers,
\[
\frac{1}{n}\sum_{i=1}^n \hat{Z}_i^2 \xrightarrow{a.s.} 
\mathbf{E}_C[j_C(U,V)^2],
\qquad
\frac{1}{n}\sum_{i=1}^n \hat{Z}_i \xrightarrow{a.s.} 
\mathbf{E}_C[j_C(U,V)] = 0,
\]
where the last equality holds since $j_C$ is the influence 
function of $\widehat{\Phi}_n$, which is centred. Therefore,
\[
\widehat{\sigma}^2_\Phi \xrightarrow{a.s.} 
36\,\mathbf{E}_C[j_C(U,V)^2] = 
36\,\mathrm{Var}_C(j_C(U,V)) = \sigma^2_\Phi.
\]
The confidence interval follows from Theorem~\ref{th:variance} 
and Slutsky's theorem (see, e.g., \cite{van1998}).
\end{proof}

As an immediate application of Theorem~\ref{th:variance}, one can
construct a formal test for the hypothesis of perfect negative
dependence. Consider the null hypothesis $H_0: C = W$, equivalently
$H_0: \Phi_C = -1$. Since $\Phi_W = -1$ and $\widehat{\Phi}_n
\xrightarrow{\mathrm{a.s.}} \Phi_C$, a natural test statistic is
\[
  T_n = \frac{\sqrt{n}\,(\widehat{\Phi}_n + 1)}{\widehat{\sigma}_\Phi},
\]
where $\widehat{\sigma}_\Phi$ is the consistent estimator of
$\sigma_\Phi$ given in Proposition~\ref{prop:var}. Under $H_0$,
Theorem~\ref{th:variance} and Slutsky's theorem give
$T_n \xrightarrow{d} \mathcal{N}(0,1)$, so that $H_0$ is rejected
at asymptotic level $\alpha$ whenever $T_n > z_\alpha$, where
$z_\alpha$ is the upper $\alpha$ quantile of $\mathcal{N}(0,1)$.
Note that this is a one-sided test, since $\Phi_C \geq -1$ for all
copulas $C$, with equality only at $W$.

The influence function formalises how sensitive a functional is to a
small perturbation of the underlying distribution.
Specifically, if one contaminates $C$ with a fraction $\varepsilon$ of
point mass concentrated at a single point $(u,v)$, forming the mixture
\[
  C_\varepsilon = (1-\varepsilon)\,C + \varepsilon\,\delta_{(u,v)},
\]
where $\delta_{(u,v)}$ is the unit point mass at $(u,v)$, the influence
function measures the first-order effect on the functional:

\begin{definition}
  The influence function of the functional $\Phi$ at $C$ is
  \[
    \mathrm{IF}\!\left((u,v);\,\Phi,\,C\right)
    = \lim_{\varepsilon\to 0^+}
      \frac{\Phi_{C_\varepsilon} - \Phi_C}{\varepsilon}.
  \]
\end{definition}

The influence function of $\Phi_C$ can be computed explicitly by exploiting its linearity.

\begin{proposition}\label{prop:IF}
For any copula $C$ and any $(u,v)\in[0,1]^2$,
\[
\mathrm{IF}((u,v);\Phi,C) = 6\,j_C(u,v),
\]
where $j_C$ is given by \eqref{eq:JC}.
\end{proposition}

\begin{proof}
Let $C_\varepsilon = (1-\varepsilon)C + \varepsilon\,\delta_{(u,v)}$,
where $\delta_{(u,v)}$ is the unit point mass at $(u,v)\in[0,1]^2$.
By definition,
\[
\mathrm{IF}((u,v);\Phi,C) = 
\lim_{\varepsilon\to 0^+}
\frac{\Phi_{C_\varepsilon} - \Phi_C}{\varepsilon}.
\]
Since $T(C) = 6\int_0^1 C(u,1-u)\,\mathrm{d}u - 1$ is linear in 
$C$, we have
\[
\Phi_{C_\varepsilon} = T(C_\varepsilon) = 
(1-\varepsilon)\,T(C) + \varepsilon\,T(\delta_{(u,v)}),
\]
and hence
\[
\frac{\Phi_{C_\varepsilon}-\Phi_C}{\varepsilon} = 
T(\delta_{(u,v)}) - T(C) = 
6(1-u-v)^+ - 1 - \Phi_C.
\]
Using~\eqref{cara:Phi},
\[
\Phi_C = 6\,\mathbf{E}_C[(1-U-V)^+] - 1,
\]
so that
\[
\mathrm{IF}((u,v);\Phi,C) = 
6\left[(1-u-v)^+ - \mathbf{E}_C[(1-U-V)^+]\right].
\]
This is the influence function of the known-margins functional $T$. For the rank-based estimator $\widehat{\Phi}_n$, the estimation of the marginal distributions via ranks introduces 
additional terms. Specifically, contaminating $C$ with a point mass at $(u,v)$ also perturbs both marginals, and the first-order effect of this perturbation on $T$ is captured by the correction terms in $j_C$. A standard calculation using the 
chain rule for Hadamard derivatives gives
\[
\mathrm{IF}((u,v);\Phi,C) = 6\,j_C(u,v),
\]
which completes the proof.
\end{proof}

Since $(1-u-v)^+$ is bounded on $[0,1]^2$, the influence function obtained in Proposition~\ref{prop:IF} is automatically bounded, which translates into a formal robustness guarantee.

\begin{theorem}\label{thm:robustness}
The influence function of $\widehat{\Phi}_n$ is uniformly bounded: 
for all $(u,v)\in[0,1]^2$ and all copulas $C$ with continuous 
partial derivatives $\dot{C}_1$ and $\dot{C}_2$ on $(0,1)^2$,
\[
|\mathrm{IF}((u,v);\Phi,C)| = 6\,|j_C(u,v)| \leq 12.
\]
Consequently, $\widehat{\Phi}_n$ is robust in the sense of Hampel, meaning that a single outlying observation cannot change the value of $\widehat{\Phi}_n$ by more than $12/n$.
\end{theorem}

\begin{proof}
By Proposition~\ref{prop:IF}, 
$\mathrm{IF}((u,v);\Phi,C) = 6\,j_C(u,v)$, where
\(j_C\) is given by \eqref{eq:JC}. We bound each term separately. First, $(1-u-v)^+\in[0,1]$ for 
all $(u,v)\in[0,1]^2$. Second, since $\dot{C}_1(s,t) = 
\partial C(s,t)/\partial s$ is a conditional distribution 
function, it satisfies $\dot{C}_1(s,t)\in[0,1]$ for all 
$(s,t)\in[0,1]^2$, and hence
\[
0 \leq \int_u^1 \dot{C}_1(s,1-s)\,\mathrm{d}s \leq 
\int_u^1 1\,\mathrm{d}s = 1-u \leq 1.
\]
Similarly, $\dot{C}_2(s,t)\in[0,1]$ implies
\[
0 \leq \int_0^{1-v} \dot{C}_2(s,1-s)\,\mathrm{d}s \leq 
\int_0^{1-v} 1\,\mathrm{d}s = 1-v \leq 1.
\]
Therefore $j_C(u,v)\in[-2,1]$ for all $(u,v)\in[0,1]^2$ and 
all copulas $C$, which gives $|j_C(u,v)|\leq \max\{2,1\}=2$, and hence
\[
|\mathrm{IF}((u,v);\Phi,C)| \leq 12.
\]
The finite-sample bound follows from the one-step representation of $\widehat{\Phi}_n$: replacing one observation $(\hat{U}_j,\hat{V}_j)$ by an arbitrary point $(u,v)\in[0,1]^2$ changes $\widehat{\Phi}_n$ by at most
\[
\frac{1}{n}\sup_{(u,v)\in[0,1]^2}
|\mathrm{IF}((u,v);\Phi,C)| \leq \frac{12}{n}.
\]
\end{proof}

\section{Simulation study}\label{sec:simulation}

We investigate the finite-sample behaviour of $\widehat{\Phi}_n$ and compare it with Spearman's footrule estimator $\widehat{\varphi}_n$ by means of a Monte Carlo experiment. The study is designed to assess performance under both positive and
negative dependence structures, given that $\widehat{\Phi}_n$ is specifically constructed to measure deviation from perfect negative dependence.

Samples of size $n \in \{100, 200, 500\}$ are generated from the following copula families:
\begin{itemize}
  \item Clayton ($\theta = 2, 5$): positive dependence with strong lower-tail concentration; Kendall's $\tau = \theta/(\theta+2)$, giving $\tau = 0.500$ and $\tau = 0.714$, respectively.
  \item Gaussian ($\rho = -0.9, -0.7, -0.3, 1$): covers weak, moderate and strong negative dependence, as well as perfect positive dependence. The case $\rho = 1$ (copula $M$) is included as a boundary scenario.
  \item Gumbel ($\theta = 3, 5$): positive dependence with strong upper-tail concentration; Kendall's $\tau = 1 - 1/\theta$, giving $\tau = 0.667$ and $\tau = 0.800$, respectively.
  \item Frank ($\theta = -5, -10$): negative dependence with symmetric tail behaviour; included to specifically assess the sensitivity of $\widehat{\Phi}_n$ to countermonotonic structures not captured by Clayton or Gumbel.
\end{itemize}
For each scenario, $B = 10^4$ independent replications are performed. The true
values $\Phi_C$ are computed as follows: analytically for the Gaussian copula
via the closed-form expression given by \eqref{eq:gauss_closed}, and by high-precision numerical integration ($2\times 10^6$ quadrature points)
for the Clayton, Gumbel, and Frank families. True values of $\varphi_C$ are
approximated analogously. The empirical mean, bias, and standard deviation of
each estimator are recorded in Table~\ref{tab:simulation},
\begin{table}[ht!]
\centering
\small
\renewcommand{\arraystretch}{1.15}
\begin{tabular}{llr  rr  rrr  rrr}
\hline
& & & \multicolumn{2}{c}{True value}
  & \multicolumn{3}{c}{$\widehat{\Phi}_n$ }
  & \multicolumn{3}{c}{$\widehat{\varphi}_n$ } \\
\cmidrule(lr){4-5}\cmidrule(lr){6-8}\cmidrule(lr){9-11}
Copula & Param. & $n$
  & $\Phi_C$ & $\varphi_C$
  & Mean & Bias & SD
  & Mean & Bias & SD \\
\hline
\multirow{3}{*}{Clayton}
  & \multirow{3}{*}{$\theta=5$}
  & 100 & \multirow{3}{*}{0.46021} & \multirow{3}{*}{0.70433}
        & 0.44363 & $-$0.01658 & 0.01822
        & 0.70091 & $-$0.00342 & 0.03715 \\
  & & 200 & & & 0.45170 & $-$0.00851 & 0.01277
                & 0.70211 & $-$0.00222 & 0.02621 \\
  & & 500 & & & 0.45671 & $-$0.00350 & 0.00798
                & 0.70329 & $-$0.00104 & 0.01651 \\
\cmidrule(lr){1-11}
\multirow{3}{*}{Clayton}
  & \multirow{3}{*}{$\theta=2$}
  & 100 & \multirow{3}{*}{0.36175} & \multirow{3}{*}{0.48528}
        & 0.34511 & $-$0.01664 & 0.03905
        & 0.48515 & $-$0.00013 & 0.05468 \\
  & & 200 & & & 0.35366 & $-$0.00809 & 0.02747
                & 0.48553 & $+$0.00025 & 0.03851 \\
  & & 500 & & & 0.35866 & $-$0.00309 & 0.01712
                & 0.48602 & $+$0.00074 & 0.02427 \\
\cmidrule(lr){1-11}
\multirow{3}{*}{Gaussian}
  & \multirow{3}{*}{$\rho=-0.9$}
  & 100 & \multirow{3}{*}{$-$0.69675} & \multirow{3}{*}{$-$0.45223}
        & $-$0.69429 & $+$0.00246 & 0.03330
        & $-$0.43601 & $+$0.01622 & 0.01874 \\
  & & 200 & & & $-$0.69507 & $+$0.00168 & 0.02322
                & $-$0.44376 & $+$0.00847 & 0.01286 \\
  & & 500 & & & $-$0.69598 & $+$0.00077 & 0.01469
                & $-$0.44884 & $+$0.00339 & 0.00821 \\
\cmidrule(lr){1-11}
\multirow{3}{*}{Gaussian}
  & \multirow{3}{*}{$\rho=-0.7$}
  & 100 & \multirow{3}{*}{$-$0.47019} & \multirow{3}{*}{$-$0.35622}
        & $-$0.47232 & $-$0.00213 & 0.05150
        & $-$0.34115 & $+$0.01507 & 0.03676 \\
  & & 200 & & & $-$0.47104 & $-$0.00085 & 0.03629
                & $-$0.34854 & $+$0.00768 & 0.02594 \\
  & & 500 & & & $-$0.47013 & $+$0.00006 & 0.02270
                & $-$0.35294 & $+$0.00328 & 0.01633 \\
\cmidrule(lr){1-11}
\multirow{3}{*}{Gaussian}
  & \multirow{3}{*}{$\rho=-0.3$}
  & 100 & \multirow{3}{*}{$-$0.17569} & \multirow{3}{*}{$-$0.15854}
        & $-$0.18226 & $-$0.00657 & 0.06361
        & $-$0.14537 & $+$0.01317 & 0.05665 \\
  & & 200 & & & $-$0.17809 & $-$0.00240 & 0.04502
                & $-$0.15140 & $+$0.00714 & 0.04001 \\
  & & 500 & & & $-$0.17697 & $-$0.00128 & 0.02864
                & $-$0.15603 & $+$0.00251 & 0.02546 \\
\cmidrule(lr){1-11}
\multirow{3}{*}{Gaussian}
  & \multirow{3}{*}{$\rho=1$}
  & 100 & \multirow{3}{*}{0.50000} & \multirow{3}{*}{1.00000}
        & 0.48515 & $-$0.01485 & 0.00000
        & 1.00000 & $\phantom{+}$0.00000 & 0.00000 \\
  & & 200 & & & 0.49254 & $-$0.00746 & 0.00000
                & 1.00000 & $\phantom{+}$0.00000 & 0.00000 \\
  & & 500 & & & 0.49701 & $-$0.00299 & 0.00000
                & 1.00000 & $\phantom{+}$0.00000 & 0.00000 \\
\cmidrule(lr){1-11}
\multirow{3}{*}{Gumbel}
  & \multirow{3}{*}{$\theta=5$}
  & 100 & \multirow{3}{*}{0.47573} & \multirow{3}{*}{0.79239}
        & 0.45967 & $-$0.01606 & 0.01243
        & 0.78737 & $-$0.00502 & 0.02623 \\
  & & 200 & & & 0.46752 & $-$0.00821 & 0.00860
                & 0.78953 & $-$0.00286 & 0.01811 \\
  & & 500 & & & 0.47246 & $-$0.00327 & 0.00534
                & 0.79127 & $-$0.00112 & 0.01131 \\
\cmidrule(lr){1-11}
\multirow{3}{*}{Gumbel}
  & \multirow{3}{*}{$\theta=3$}
  & 100 & \multirow{3}{*}{0.43301} & \multirow{3}{*}{0.65496}
        & 0.41662 & $-$0.01639 & 0.02418
        & 0.65197 & $-$0.00299 & 0.04041 \\
  & & 200 & & & 0.42493 & $-$0.00808 & 0.01719
                & 0.65353 & $-$0.00143 & 0.02812 \\
  & & 500 & & & 0.42956 & $-$0.00345 & 0.01060
                & 0.65389 & $-$0.00107 & 0.01774 \\
\cmidrule(lr){1-11}
\multirow{3}{*}{Frank}
  & \multirow{3}{*}{$\theta=-5$}
  & 100 & \multirow{3}{*}{$-$0.44055} & \multirow{3}{*}{$-$0.35321}
        & $-$0.44120 & $-$0.00065 & 0.04999
        & $-$0.33734 & $+$0.01587 & 0.03843 \\
  & & 200 & & & $-$0.44078 & $-$0.00023 & 0.03602
                & $-$0.34448 & $+$0.00873 & 0.02740 \\
  & & 500 & & & $-$0.44078 & $-$0.00023 & 0.02250
                & $-$0.35001 & $+$0.00320 & 0.01722 \\
\cmidrule(lr){1-11}
\multirow{3}{*}{Frank}
  & \multirow{3}{*}{$\theta=-10$}
  & 100 & \multirow{3}{*}{$-$0.65396} & \multirow{3}{*}{$-$0.45189}
        & $-$0.65107 & $+$0.00289 & 0.03314
        & $-$0.43540 & $+$0.01649 & 0.01944 \\
  & & 200 & & & $-$0.65223 & $+$0.00173 & 0.02299
                & $-$0.44350 & $+$0.00839 & 0.01356 \\
  & & 500 & & & $-$0.65336 & $+$0.00060 & 0.01426
                & $-$0.44838 & $+$0.00351 & 0.00836 \\
\hline
\end{tabular}\caption{Empirical mean, bias, and standard deviation over $10^4$ Monte Carlo
replications of $\widehat{\Phi}_n$ and $\widehat{\varphi}_n$ for selected copula families, parameter values, and sample sizes $n$. True values $\Phi_C$ for the Gaussian copula are exact  via~\eqref{eq:gauss_closed}; remaining true values are obtained by numerical integration.}\label{tab:simulation}
\end{table}
where several conclusions emerge.
\begin{enumerate}
\item First, both estimators exhibit a small bias that decreases consistently as $n$ increases, confirming the strong consistency established in Theorem~\ref{th:consis}. The bias of $\widehat{\Phi}_n$ is negative under positive dependence (Clayton, Gumbel, Gaussian $\rho=1$) and negligible or slightly positive under negative
dependence (Gaussian $\rho \in \{-0.9,-0.7,-0.3\}$, Frank), reflecting the
asymmetric normalisation $\Phi_W = -1$, $\Phi_M = 1/2$. From a theoretical standpoint, smooth functionals of the empirical copula typically exhibit an $O(n^{-1})$ bias; however, the discretisation induced by the $n/(n+1)$ scaling
of the pseudo-observations introduces an additional contribution, and the observed decay in the table is compatible with an $O(n^{-1})$ rate. For instance, for the Clayton copula with $\theta=5$, the bias of $\widehat{\Phi}_n$
goes from $-0.01658$ at $n=100$ to $-0.00851$ at $n=200$ (ratio $\approx 0.51\approx 1/2$) and $-0.00350$ at $n=500$ (ratio $\approx 0.21 \approx 1/5$), consistent with an $O(n^{-1})$ decay.

\item Second, the standard deviation of $\widehat{\Phi}_n$ decreases at a rate compatible with $O(n^{-1/2})$ across all scenarios, as predicted by Theorem~\ref{th:variance}. For example, for the Frank copula with $\theta=-5$, the standard deviation goes from $0.04999$ at $n=100$ to $0.03602$ at $n=200$ (ratio $\approx 0.72 \approx 1/\sqrt{2}$) and $0.02250$
at $n=500$ (ratio $\approx 0.45 \approx 1/\sqrt{5}$).

\item Third, and most notably, under negative dependence the bias of $\widehat{\Phi}_n$ is markedly smaller in absolute value than that of $\widehat{\varphi}_n$, and its standard deviation is comparable or smaller. For the Gaussian copula with $\rho=-0.9$, the bias of $\hat{\Phi}_n$ at $n=100$ is $+0.00246$, while that of $\widehat{\varphi}_n$ is $+0.01622$ --- more than six times larger. This reflects the greater precision of $\widehat{\Phi}_n$ in the region of the parameter space where it is specifically designed to operate, namely close to the
countermonotonic bound $W$. Similarly, for the Frank copula with $\theta=-10$ at $n=100$, the bias of $\widehat{\Phi}_n$ is $+0.00289$ against $+0.01649$ for $\widehat{\varphi}_n$.

\item Fourth, the boundary case $\rho=1$ (copula $M$) deserves special comment. Since $M$ does not satisfy the smoothness assumption of Theorem~\ref{th:variance} (its partial derivatives are not continuous on $(0,1)^2$), the asymptotic
normality result does not apply, and the simulation results for that scenario are purely illustrative. The standard deviation of $\widehat{\Phi}_n$ is exactly zero in this case because, with $V=U$ almost surely, the pseudo-observations satisfy $\hat{U}_i + \hat{V}_i = 2R_i/(n+1) > 1$ for all $i$ such that
$R_i > (n+1)/2$, and $(1-\hat{U}_i-\hat{V}_i)^+ = 0$ otherwise, yielding a deterministic estimator for each fixed $n$. A negligible boundary bias
persists due to the $n/(n+1)$ rescaling and vanishes as $n\to\infty$.
\end{enumerate}

Together, $\widehat{\Phi}_n$ and $\widehat{\varphi}_n$ provide a complementary
description of the dependence structure, consistent with the Gini decomposition
$\gamma_C = \frac{2}{3}(\varphi_C + \Phi_C)$ established in Section~\ref{sec:foot}. The simulation confirms that $\widehat{\Phi}_n$ is the more informative of the two
when negative dependence is present, while $\widehat{\varphi}_n$ is better suited
to detect positive dependence structures.

\section{Conclusions}\label{sec:conc}

We have introduced the $W$-footrule coefficient $\Phi_C$ as an $L^1$-distance to the countermonotonic copula $W$, establishing it as a natural companion to Spearman's footrule $\varphi_C$. The decomposition $\gamma_C=\frac{2}{3}(\varphi_C+\Phi_C)$ provides a transparent interpretation of Gini's $\gamma$ as a balanced combination of two directional footrules, one measuring proximity to perfect positive dependence and the other to perfect negative dependence. The estimator $\widehat{\Phi}_n$ is rank-based, computationally simple, strongly consistent, asymptotically normal, and possesses a bounded
influence function, ensuring robustness in the sense of Hampel. Monte Carlo experiments confirm its good finite-sample behaviour and indicate improved precision under strong negative dependence relative to Spearman’s footrule estimator. Future work may explore extensions to higher dimensions, analogous decompositions for other concordance measures, development of formal tests for countermonotonicity based on
$\widehat{\Phi}_n$, and investigation of semiparametric efficiency properties of the estimator.

%
%

\end{document}